\documentclass[reqno]{amsart}
\usepackage{amssymb}
\usepackage{pdfsync}
\usepackage{graphicx}
\pdfoutput=1

\numberwithin{equation}{section}

\newtheorem{theorem}{Theorem}[section]

\newtheorem{corollary}[theorem]{Corollary}

\theoremstyle{definition}

\theoremstyle{definition} 

\def\no{|\!|}

\newcommand{\cB}{\mathcal B}
\newcommand{\cE}{\mathcal E}

\newcommand{\cM}{\mathcal M}
\newcommand{\R}{\mathbb R}
\newcommand{\C}{\mathbb C}

\newcommand{\N}{\mathbb N}
\newcommand{\LL}{\mathbb L}

\newcommand{\st}{\,:\,}

\newcommand{\EE}{\mathbb E}
\newcommand{\FF}{\mathbb F}
\newcommand{\YY}{\mathbb Y}

\DeclareMathSymbol{\complement}{\mathord}{AMSa}{"7B}

\def\vv<#1>{\langle #1\rangle}
\def\Vv<#1>{\bigl\langle #1\bigr\rangle}

\begin{document}


\title[Foliations for quasilinear parabolic problems]
{Invariant foliations near normally hyperbolic equilibria for quasilinear parabolic problems}

\author[J. Pr\"uss]{Jan Pr\"uss}
\address{Martin-Luther-Universit\"at Halle-Witten\-berg\\
         Institut f\"ur Mathematik \\
         Theodor-Lieser-Strasse 5\\
         D-06120 Halle, Germany}
\email{jan.pruess@mathematik.uni-halle.de}

\author[G. Simonett]{Gieri Simonett}
\address{Department of Mathematics\\
         Vanderbilt University \\
         Nashville, TN~37240, USA}
\email{gieri.simonett@vanderbilt.edu}

\author[M. Wilke]{Mathias Wilke}
\address{Martin-Luther-Universit\"at Halle-Witten\-berg\\
         Institut f\"ur Mathematik \\
         Theodor-Lieser-Strasse 5\\
         D-06120 Halle, Germany}
\email{mathias.wilke@mathematik.uni-halle.de}

\thanks{}

\begin{abstract}
We consider quasilinear
parabolic evolution equations in the situation where the set of
equilibria forms a finite-dimensional $C^1$-manifold which is normally hyperbolic.
The existence of foliations of the stable and unstable manifolds is shown assuming merely
$C^1$-regularity of the underlying equation.
\end{abstract}
\maketitle

\\
\noindent
\section{Introduction}
In recent years, foliations of invariant manifolds of (semi)-flows have attracted considerable attention. They exhibit a precise analysis of the given (semi)-flow near a certain special invariant manifold $\cE$,
most often a center manifold.
We refer to \cite{BLZ98, BLZ00, BLZ08} for a comprehensive discussion of the subject
for infinite-dimensional dynamical systems,
and to \cite{HPS77,W94} for the finite-dimensional case.

In this paper we consider a special situation, namely the case where $\cE$ is a
$C^1$-manifold of equilibria for a quasilinear parabolic system. The precise setting for this paper is as follows.
Let $X_0$ and $X_1$ be two Banach spaces such that
$X_1$ is continuously and densely
embedded in $X_0$. We then consider the abstract autonomous
quasilinear problem
\begin{equation}
\label{u-equation}
\dot{u}(t)+A(u(t))u(t)=F(u(t)),\quad t>0, \quad u(0)=u_0.
\end{equation}
For $1<p<\infty$ we introduce the real interpolation space
$X_\gamma:=(X_0,X_1)_{1-1/p,p}$ and we assume that there is an open
set $V\subset X_\gamma$ such that
\begin{equation}
\label{AF}
(A,F)\in C^1(V,\cB(X_1,X_0)\times X_0).
\end{equation}
Here $\cB(X_1,X_0)$ denotes the space of all bounded
linear operators from $X_1$ into $X_0$.
In the sequel we use the notation $|\cdot|_j$ to denote the norm
in the respective spaces $X_j$ for $j\in\{0,1,\gamma\}$.
Moreover, for any normed space $X$,
$B_X(u,r)$ denotes the open ball in $X$ with radius $r>0$ around $u\in X$.
\smallskip\\
\noindent
If $A(u)$ has the property of {\bf maximal $L_p$-regularity} for each $u\in V$, then it is well-known that problem \eqref{u-equation} generates a local semiflow in $V\subset X_\gamma$;
see e.g.
\ \cite{KPW10, Pru03}.
Let $ \cE\subset V\cap X_1$ denote the set of  equilibrium solutions of (\ref{u-equation}), which means that
$$
u\in\cE \quad \mbox{ if and only if }\quad u\in V\cap X_1,
\; A(u)u=F(u).
$$
Given an  element $u_*\in\cE$,  we consider the situation that $u_*$ is
contained in an $m$-dimensional manifold of equilibria. This means that there
is an open subset $U\subset\R^m$, $0\in U$, and a $C^1$-function
$\Psi:U\rightarrow X_1$,  such that
\begin{equation}
\label{manifold}
\begin{aligned}
& \bullet\
\text{$\Psi(U)\subset \cE$ and $\Psi(0)=u_*$,} \\
& \bullet\
 \text{the rank of $\Psi^\prime(0)$ equals $m$, and} \\
& \bullet\
\text{$A(\Psi(\zeta))\Psi(\zeta)=F(\Psi(\zeta)),\quad \zeta\in U.$}
\end{aligned}
\end{equation}
Let $A_0:= A(u_*) +[A^\prime(u_*)\,\cdot\,]u_* -F^\prime(u_*)$ be the linearization of \eqref{u-equation} at $u_*\in\cE$.
An equilibrium  $u_*$ is called {\bf normally hyperbolic} if
\begin{itemize}
\item[(i)] near $u_*$ the set of equilibria $\cE$ is a $C^1$-manifold in $X_1$ of dimension $m\in\N$,
\vspace{-4mm}
\item[(ii)] \, the tangent space for $\cE$ at $u_*$ is given by $N(A_0)$,
\item[(iii)] \, $0$ is a semi-simple eigenvalue of
$A_0$, i.e.\ $ N(A_0)\oplus R(A_0)=X_0$,
\item[(iv)] \, $\sigma(A_0)\cap i\R =\{0\}$, $\sigma_u:=\sigma(A_0)\cap \C_-\neq\emptyset$.
\end{itemize}
An equilibrium is called {\bf normally stable} if (i), (ii), (iii) hold and (iv) is replaced by
$$(iv)_s \qquad  \sigma(A_0)\cap i\R =\{0\},\quad \sigma(A_0)\setminus\{0\}\subset \C_+,$$
which means that $\sigma_u=\emptyset$.
Note that a standard argument  in the theory of Fredholm operators implies that $R(A_0)$ is closed
and $A_0$ is a Fredholm operator of index 0, even if (iii) holds only algebraically,
see for instance \cite[p.78]{D85}. So in fact (iii) is also true topologically, and $\{0\}$ is an isolated point in $\sigma(A_0)$.

In the recent paper \cite{PSZ09} it is proved that any solution $u(t)$ starting near a normally hyperbolic $u_*\in\cE$ and staying near $\cE$ exists globally and converges in $X_\gamma$ to an equilibrium $u_\infty\in\cE$ which, however, may be different from $u_*$. In case $u_*$ is normally stable this is true for any solution starting near $u_*$. We call this the {\bf generalized principle of linearized stability}.
In that paper it is also shown that near $u_*$ there are no other equilibria
than those given by $\Psi(U)$,
i.e.\ $\cE\cap B_{X_1}(u_*,{r_1})=\Psi(U)$, for some $r_1>0$.
\smallskip

These situations appear frequently in applications, for example in problems with symmetries,
and in problems with moving boundaries,
see  for instance \cite{ES98a, ES98b, KPW12, PrSi08, PSW11, PSZ09, PSZ09b, PSZ10, PSZ12, Si01}.

Our intention in this paper is to study the behavior of the semiflow near $\cE$ in more detail. If $u_*$ is normally hyperbolic, then any $w\in\cE$ close to $u_*$ in $X_\gamma$ will be normally hyperbolic as well.
Therefore, intuitively, at each point $w\in\cE$ near $u_*$ there should be a stable manifold $\cM^s_w$ and an unstable manifold $\cM^u_w$ such that
$\cM^s_w\cap\cM^u_w\cap B(u_*,r)=\{w\}$, and these manifolds should depend continuously on $w\in B(u_*,r)\cap\cE$. The tangent spaces of these manifolds are expected to be the projections
with respect to the stable resp.\ unstable part of the spectrum of the linearization of \eqref{u-equation} at $w$.
\begin{figure}
    \centering
        \includegraphics[width=7cm, height=6cm]{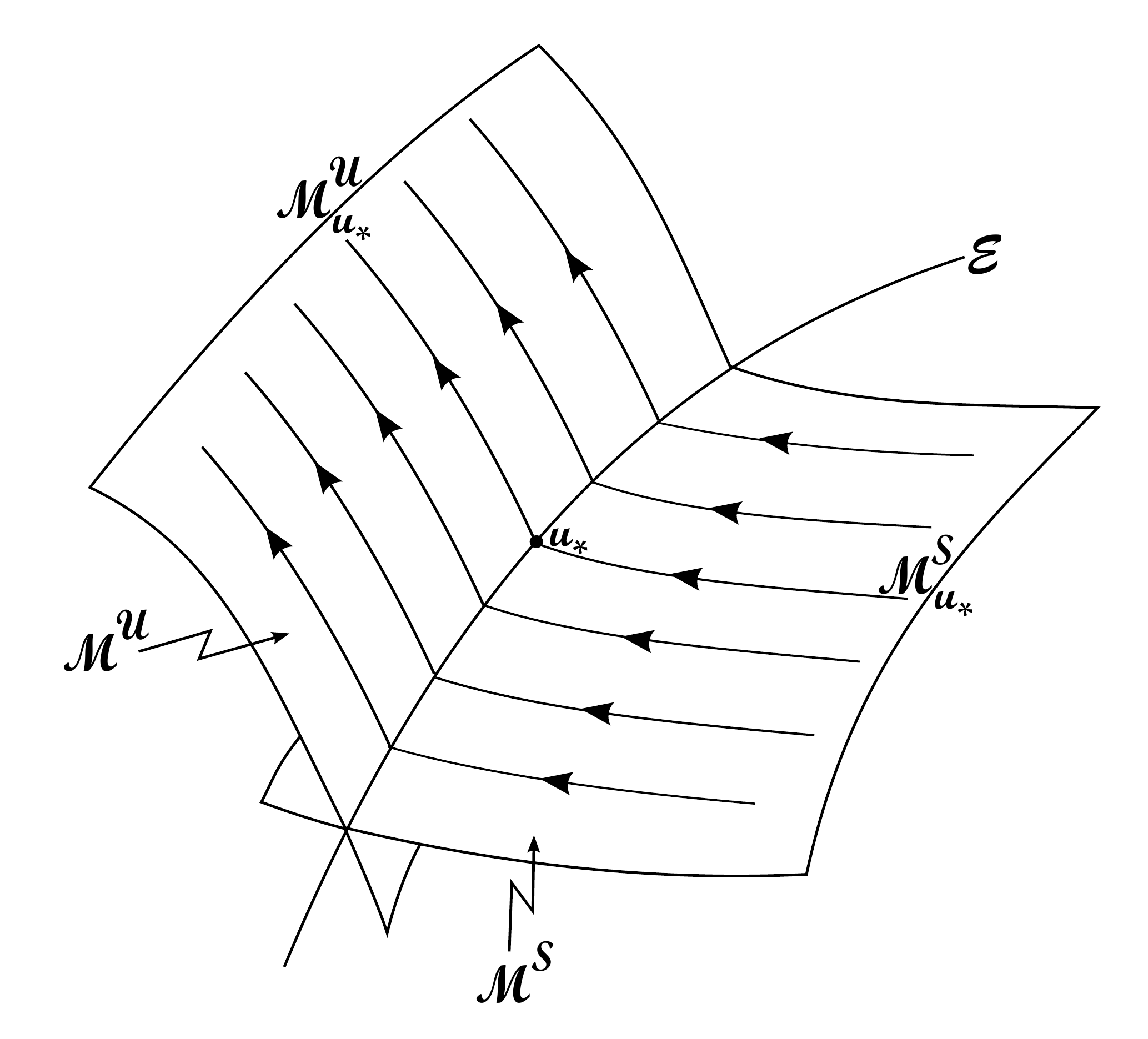}
          \caption{Foliation and fibers near $\mathcal{E}$.}
\end{figure}

We  prove these assertions below, and call it the {\bf stable} resp.\ {\bf unstable foliation} of \eqref{u-equation} near $u_*\in\cE$. The stable resp.\ unstable manifolds $\cM^s_w$ and $\cM^u_w$ are termed the {\bf leaves} or {\bf fibers} of these foliations. They turn out to be  positively and also negatively invariant under the semiflow. As a consequence, the convergent solutions are precisely those which start at initial values sitting on one of the stable fibers. In the normally stable case where $\sigma(A_0)\cap \C_-=\emptyset$, this implies that the stable foliation covers a neighborhood of $u_*$ in $V$.

The proofs we present here are quite simple and elementary, as they just employ the implicit function theorem, which is at our disposal thanks to maximal regularity.
We only need to require $C^1$-regularity for $(A,F)$ and for $\cE$,
in contrast to other work, where more regularity (at least $C^2$) is assumed.
We mention that the authors in \cite{BLZ00} announced a proof
for the $C^1$ case in a general context, see page 4644.

The fibers of these foliations will be manifolds of class $C^1$ as well, however, their dependence on $w\in\cE$ is only continuous: here we lose one degree of regularity. This is quite natural and should be compared to the loss of regularity which the normal field $\nu$ of a hyper-surface $\Gamma\subset\R^n$ of class $C^k$, $k\in\N$, suffers: it is only of class $C^{k-1}$. We also show that in case $(A,F)\in C^k$ the dependence on $w$ is of class $C^{k-1}$, for each $k\in\N\cup\{\infty,\omega\}$, where $\omega$ means real analytic.

\section{Preliminaries and the asymptotic normal form}
(a) Suppose that
the operator $A(u_*)$ has the property of maximal $L_p$-regularity.
Introducing  the deviation $v=u-u_*$ from the equilibrium $u_*$, the equation
for $v$ then reads as
\begin{equation}
\label{v-equation}
\dot{v}(t)+A_0v(t)=G(v(t)),\quad t>0,\quad v(0)=v_0,
\end{equation}
where
$v_0=u_0-u_*$
and
\begin{equation}
\label{A0}
A_0w=A(u_*)w+(A^\prime(u_*)w)u_*-F^\prime(u_*)w\quad\hbox{for }
w\in X_1.
\end{equation}
The function $G$ can be written as $G(v)=G_1(v)+G_2(v,v)$, where
\begin{equation*}
\begin{aligned}
G_1(v)&=(F(u_*+v)\!-\!F(u_*)\!-\!F^\prime(u_*)v)
\!-\!(A(u_*+v)\!-\!A(u_*)\!-\!A^\prime(u_*)v)u_*, \\
G_2(v,w)&=-(A(u_*+v)\!-\!A(u_*))w,\quad  w\in X_1,\;v\in V_\ast,
\end{aligned}
\end{equation*}
and $V_\ast:=V-u_\ast$. It follows from \eqref{AF}
that
$G_1\in C^1(V_\ast,X_0)$ and also that
$G_2\in C^1(V_\ast\times X_1,X_0)$.
Moreover, we have
\begin{equation}
\label{G(0)}
G_1(0)=G_2(0,0)=0,
\quad G_1^\prime(0)=G_2^\prime(0,0)=0,
\end{equation}
where $G^\prime_1$ and $G^\prime_2$ denote the Fr\'{e}chet
derivatives of $G_1$ and $G_2$, respectively.
\medskip\\
Setting
$\psi(\zeta)=\Psi(\zeta)-u_*$ results in the following equilibrium
equation for problem \eqref{v-equation}:
\begin{equation}
\label{equilibrium-psi}
A_0\psi(\zeta)=G(\psi(\zeta)),\quad \mbox{ for all }\;\zeta\in U.
\end{equation}
Taking the derivative with respect to $\zeta$ and using
the fact that $G^\prime(0)=0$ we conclude that
$A_0\psi^\prime(0)=0$ and this implies that
\begin{equation}
\label{tangent-space}
T_{u_\ast}(\cE)\subset N(A_0),
\end{equation}
where $T_{u_\ast}(\cE) $ denotes the tangent space of $\cE$ at
$u_\ast$.
Note that assumption (ii) yields equality in \eqref{tangent-space}, and (iii) implies that $0$ is an isolated
spectral point of $A_0$. According to
assumption (iv), $\sigma(A_0)$ admits a decomposition into three
disjoint  parts with
\begin{equation*}
\sigma(A_0)=\{0\}\cup \sigma_s \cup \sigma_u, \quad
-\sigma_u\cup\sigma_s\subset \C_+=\{z\in\C:\,{\rm Re}\, z>0\}.
\end{equation*}
The spectral set $\sigma_c:=\{0\}$ corresponds to the center part,
$\sigma_s$ to the stable one, and $\sigma_u$ to the unstable part of the analytic
$C_0$-semigroup $\{e^{-A_0 t}:t\ge 0\}$, or equivalently
of the Cauchy problem $\dot{w}+A_0w=f$.
\smallskip\\
\noindent In the normally stable case we have $\sigma_u=\emptyset$. Subsequently, we let $P^l$, $l\in\{c,s,u\}$, denote the
spectral projections according to the spectral sets $\sigma_c=\{0\}$, $\sigma_s$, and $\sigma_u$, and we set $X^l_j:=P^l X_j$ for $l\in\{c,s,u\}$ and
$j\in\{0,1,\gamma\}$. The spaces $X^l_j$ are equipped with the
norms $|\cdot| _j$ for $j\in\{0,1,\gamma\}$. We have the topological
direct decomposition
\begin{equation*}
X_1=X_1^c\oplus X_1^s\oplus X_1^u,\quad X_0=X_0^c\oplus X_0^s\oplus X_0^u,
\end{equation*}
and this decomposition reduces $A_0$ into $A_0=A_c\oplus A_s\oplus A_u$,
where  $A_l$  is the part of $A_0$ in $X^l_0$
for $l\in\{c,s,u\}$.
Since $\sigma_c=\{0\}$ and $\sigma_u$ is compact it follows that
$X_0^c\subset X_1$ and $X_0^u\subset X_1$.
Therefore, $X_0^l$ and $X_1^l$ coincide as
vector spaces for $l\in\{c,u\}$. In the following, we will
just write $X^l=(X^l,|\cdot|_l)$
for either of the spaces $X^l_0$ and $X^l_1$, $l\in\{c,u\}$.
We note that $X^c=N(A_0)$, and hence $A_c=0$.

The operator $A_s$
inherits the property of $L_p$-maximal regularity from $A_0$.
Since $\sigma(A_s)=\sigma_s\subset \C_+$,
we obtain that the Cauchy problem
\begin{equation}
\label{Cauchy}
\dot w + A_s w=f,\quad w(0)=0,
\end{equation}
also enjoys the property of maximal regularity,
even on the interval $J=(0,\infty)$.
In fact the following estimates are true: for any $a\in (0,\infty]$ let
\begin{equation}
\label{EE}
\EE_0(a)=L_p((0,a);X_0),
\quad \EE_1(a)=H^1_p((0,a);X_0)\cap L_p((0,a);X_1).
\end{equation}
The natural norms in $\EE_j(a)$ will be denoted by
$|\!|\cdot|\!|_{\EE_j(a)}$ for $j=0,1$. Then the  Cauchy problem \eqref{Cauchy} has for
each
 $f\in L_p((0,a);X_0^s)$ a unique solution
$$
w\in H^1_p((0,a);X^s_0)\cap L_p((0,a);X^s_1),
$$
and there exists a constant $M_0$ such that $|\!|w|\!|_{\EE_1(a)}\le M_0 |\!|f |\!|_{\EE_0(a)}$
for every $a>0$, and
every function $f\in L_p((0,a);X^s_0)$. In fact, since
$\sigma(A_s-\omega)$ is still contained in $\C_+$ for
$$0<\omega<\inf {\rm Re}\,\sigma_s$$
fixed, we see that the operator $(A_s-\omega)$ enjoys the same
properties as $A_s$. Therefore, every solution of the Cauchy problem
$\eqref{Cauchy}$ satisfies the estimate
\begin{equation}
\label{M0}
|\!|e^{\sigma t} w |\!|_{\EE_1(a)}
\le M_0 |\!|e^{\sigma t}f |\!|_{\EE_0(a)},
\quad \sigma\in [0,\omega], \quad a>0,
\end{equation}
for $f\in L_p((0,a);X^s_0)$, where $M_0=M_0(\omega)$ for $\omega>0$
chosen as above, see \cite[Theorem 2.4]{Do93} or \cite[Sec.\ 6]{Pru03}.
Furthermore, there exists a constant $M_1>0$ such that
\begin{equation}
\label{M1}
|\!|e^{\omega t}e^{-A_s t}P^su|\!|_{\EE_1(a)}
+\sup_{t\in [0,a)}|e^{\omega t}e^{-A_st}P^su|_\gamma\le M_1|P^su|_\gamma\,
\end{equation}
for every $u\in X_\gamma$ and $a\in (0,\infty]$.
For future use we note that
\begin{equation}
\label{zero-trace}
\sup_{t\in [0,a)}|w(t)|_\gamma\le c_0
\no w\no_{\EE_1(a)}
\quad\text{for all $w\in \EE_1(a)$ with $w(0)=0$}
\end{equation}
with a constant $c_0$ that is independent of $a\in (0,\infty]$.
\medskip\\
\noindent
In the following we assume that
$$0<\omega<\min\{\inf {\rm Re}\,\sigma_s, -\sup{\rm Re}\, \sigma_u\}.$$
Turning our attention to $A_u$, we observe that $A_u$ is bounded, and hence
$e^{-tA_u}$ extends to a group $\{e^{-A_u t}:t\in\R\}$ on  $X^u$ which satisfies
the estimate
\begin{equation}\label{Au-bound}
|e^{-A_u t}|\leq M_\omega e^{\omega t},\quad t\leq 0.
\end{equation}
We emphasize that $L_p$-maximal regularity of $A(u_*)$ implies by standard perturbation theory
that $A(u)$ has this property as well, for each $u$ in some ball
$\bar{B}_{X_\gamma}(u_\ast,\rho_0)$, and so all estimates in this paragraph are uniform w.r.t.\ $u\in \bar{B}_{X_\gamma}(u_\ast,\rho_0)$.\smallskip\\
\smallskip
\noindent
(b)
Let us now consider the mapping
\begin{equation*}
g:U\subset \R^m\to X^c,\quad
g(\zeta):=P^c\psi(\zeta),\quad \zeta\in U.
\end{equation*}
It follows from our assumptions that
$g^\prime(0)=P^c\psi^\prime(0):\R^m\to X^c$
is an isomorphism
(between the finite dimensional spaces $\R^m$ and $X^c$).
By the inverse function theorem, $g$ is a
$C^1$-diffeomorphism of a neighborhood of $0$ in $\R^m$
into a neighborhood, say $B_{X^c}(0,\rho_0)$, of $0$ in $X^c$.
Let $g^{-1}:B_{X^c}(0,\rho_0)\to U$ be the inverse mapping.
Then $g^{-1}:B_{X^c}(0,\rho_0)\to U$ is $C^1$ and $g^{-1}(0)=0$.
Next we set
$\Phi(x):=\psi(g^{-1}(x))$ for $x\in B_{X^c}(0,\rho_0)$
and we note that
\begin{equation*}
\Phi\in C^1(B_{X^c}(0,\rho_0),X_1), \quad
\Phi(0)=0,
\quad \{\Phi(x)+u_\ast \st x\in B_{X^c}(0,\rho_0)\}=\cE\cap W,
\end{equation*}
where $W$ is an appropriate neighborhood of $u_\ast$ in $X_1$.
One readily verifies that
\begin{equation*}
P^c \Phi(x)=((P^c\circ \psi)\circ g^{-1})(x)=
(g\circ g^{-1})(x)=x,\quad x\in  B_{X^c}(0,\rho_0),
\end{equation*}
and this yields
$\Phi(x)=P^c\Phi(x)+P^s\Phi(x)+P^u\Phi(x)=x+P^s\Phi(x)+P^u\Phi(x)$ for
$x\in B_{X^c}(0,\rho_0)$.
Setting $\phi_l(x):=P^l\Phi(x)$, $l\in\{s,u\}$, we conclude that
\begin{equation}
\label{phi}
\phi_j\in C^1(B_{X^c}(0,\rho_0),X_1^j),\quad \phi_j(0)=\phi_j^\prime (0)=0,
\end{equation}
and with $\phi:=\phi_s+\phi_u$ that
$$
\{x+\phi(x)+u_\ast \st x\in B_{X^c}(0,\rho_0)\}=\cE\cap W.
$$
This shows that the manifold $\cE$
can be represented
as the (translated) graph of the function $\phi$ in a neighborhood
of $u_\ast$. Moreover,
the tangent space of $\cE$ at $u_\ast$
coincides with $N(A_0)=X^c$.
By applying the projections $P^l$, $l\in\{c,s,u\}$, to equation \eqref{equilibrium-psi}
and using that $x+\phi(x)=\psi(g^{-1}(x))$
for $x\in B_{X^c}(0,\rho_0)$, and that $A_c\equiv 0$,
we obtain the following equivalent system of equations
for the equilibria of \eqref{v-equation}:
\begin{equation}
\label{equilibria-phi}
\begin{split}
&P^cG(x+\phi(x))=0, \\
&P^l G(x+\phi(x))=A_l\phi_l(x),
\quad x\in B_{X_c}(0,\rho_0),\;\; l\in\{s,u\}.
\end{split}
\end{equation}
(c) Introducing the new variables
\begin{equation*}
\begin{aligned}
&x=P^c v-\xi=P^c (u-u_*)-\xi, \\
&y=P^sv-\phi_s(\xi)=P^s(u-u_*)-\phi_s(\xi),\\
&z=P^uv-\phi_u(\xi)=P^u(u-u_*)-\phi_u(\xi),\\
\end{aligned}
\end{equation*}
where $\xi\in B_{X^c}(0,\rho_0)$, we
obtain the following system of evolution equations
in $X^c\times X^s_0\times X^u$:
\begin{equation}
\label{system}
\left\{
\begin{aligned}
\dot{x}=R_c(x,y,z,\xi),      \qquad &x(0)=x_0-\xi, \\
\dot{y}+A_sy=R_s(x,y,z,\xi), \qquad &y(0)=y_0-\phi_s(\xi),\\
\dot{z}+A_uz=R_u(x,y,z,\xi), \qquad &z(0)=z_0-\phi_u(\xi),\\
\end{aligned}
\right.
\end{equation}
with $x_0=P^cv_0$, $y_0=P^sv_0$, $z_0=P^uv_0$, and
\begin{equation*}
\begin{aligned}
&R_c(x,y,z,\xi)=P^c G(x+y+z+\xi+\phi(\xi)), \\
&R_l(x,y,z,\xi)=P^lG(x+y+z+\xi+\phi(\xi))-A_l\phi_l(\xi)\\
\end{aligned}
\end{equation*}
for $l\in\{s,u\}$.
Using the equilibrium equations \eqref{equilibria-phi}, the
expressions for $R_l$ can be rewritten as
\begin{equation}
\label{R-T}
\begin{aligned}
&R_l(x,y,z,\xi)=P^l\big(G(x+y+z+\xi+\phi(\xi))-G(\xi+\phi(\xi))\big),
\end{aligned}
\end{equation}
where $l\in\{c,s,u\}$.
Although the term $P^cG(\xi+\phi(\xi))$ in $R_c$ is zero, see
\eqref{equilibria-phi}, we include it here for reasons of symmetry. Equation \eqref{R-T}
immediately yields
\begin{equation*}
\label{R=T=0}
R_c(0,0,0,\xi)=R_s(0,0,0,\xi)=R_u(0,0,0,\xi)=0\quad\text{for all }\ \xi\in B_{X^c}(0,\rho_0),
\end{equation*}
showing that
the equilibrium set $\cE$ of \eqref{u-equation}
near $u_*$ has been reduced to the set
$\{0\}\times \{0\}\times\{0\}\times B_{X^c}(0,\rho_0)\subset X^c\times X^s_1\times X^u\times X^c$.
\medskip\\
Observe also that there is a unique correspondence between the solutions of \eqref{u-equation}
close  to $u_*$ in $X_\gamma$ and those of (\ref{system}) close to $0$.
As $u_\infty:=u_*+\xi + \phi(\xi)\in\cE$ will be the limit of $u(t)$ in $X_\gamma$ as $t\to\infty$,
we call  system \eqref{system} the {\em asymptotic normal form}
of \eqref{u-equation} near
its normally stable equilibrium $u_*$.


\section{The stable foliation}
To motivate our approach to the construction of the stable foliation we formally define a map
$H_s$ according to
\begin{align}
\label{stableF}
H_s((x,y,z),(y_0,\xi))(t)=
\left[\begin{array}{l}
x(t)+\int_t^\infty R_c(x(\tau),y(\tau),z(\tau),\xi)\,d\tau\\
\vspace{-2mm}\\
y(t)-L_s(R_s(x,y,z,\xi),y_0-\phi_s(\xi))\\
\vspace{-2mm}\\
z(t)+\int_t^\infty e^{-A_u(t-\tau)}R_u(x(\tau),y(\tau),z(\tau),\xi)\,d\tau \\
\end{array}\right],
\end{align}
where $t>0$.
Here $w=L_s(f,y_0)$ denotes the unique solution of the problem
$$\dot{w}(t)+A_sw(t)=f(t),\; t>0,\quad w(0)=y_0.$$
Obviously, we have $H_s(0,0)=0$. Moreover, $H_s$ will be of class $C^1$ w.r.t.\ to the
variables $(x,y,z,y_0)$, but in general only continuous in $\xi$.
The derivative of $H_s$ w.r.t.\ $(x,y,z)$ at (0,0) is given by
$ D_{(x,y,z)}H_s(0,0)= I$, and hence the implicit function theorem formally applies and yields
a map
\begin{equation}
\label{Xi-s}
\Lambda^s:(y_0,\xi)\mapsto (x,y,z)
\end{equation}
which is well-defined near 0, such that $H_s(\Lambda^s(y_0,\xi),(y_0,\xi))=0$.
$\Lambda^s$ will be continuous in $(y_0,\xi)$, and of class $C^1$ w.r.t.\ $y_0$.
Given $(x,y,z)=\Lambda^s(y_0,\xi)$, we set
\begin{equation}
\label{IV}
\begin{split}
x_0:=& -\int_0^\infty R_c(x(\tau),y(\tau),z(\tau),\xi)\,d\tau + \xi,\\
z_0:=& -\int_0^\infty e^{A_u\tau}R_u(x(\tau),y(\tau),z(\tau),\xi)\,d\tau + \phi_u(\xi).
\end{split}
\end{equation}
For $(x,y,z)=\Lambda^s(y_0,\xi)$ given,
the first component of $H_s$ yields
$$\dot{x}(t)=R_c(x(t),y(t),z(t),\xi),\quad t\geq0,$$
 the second component implies
 $$\dot{y}(t)+A_sy(t)=R_s(x(t),y(t),z(t),\xi), \quad t>0,\quad y(0)=y_0-\phi_s(\xi),$$
and the third one leads to
$$\dot{z}(t)+A_uz(t)=R_u(x(t),y(t),z(t),\xi),\quad t>0.$$
Moreover, due to \eqref{IV}, the initial values of $x$ and $z$ are given by
\begin{equation*}
x(0)=x_0-\xi,\quad z(0)=z_0-\phi_u(\xi).
\end{equation*}
Assuming, in addition, that $(x(t),y(t),z(t))$ converges to $(0,0,0)$ as $t\to \infty$,
we conclude that
\begin{equation*}
u(t)= u_*+x(t)+y(t)+z(t)+\xi+ \phi(\xi),\quad t>0,
\end{equation*}
is a solution of \eqref{u-equation}
 with $\lim_{t\to \infty} u(t)=u_\infty:=u_*+\xi+\phi(\xi)\in\cE$.
The map
\begin{equation}
\label{xi-s}
\lambda^s: (y_0,\xi)\mapsto u(0)= u_*+x(0)+y(0)+z(0) +\xi+ \phi(\xi)
\end{equation}
yields a foliation of the stable manifold $\cM^s$ near $u_*$, and
\begin{equation}
\label{fiber-s}
\cM_{\xi}^s:=\{\lambda^s(y_0,\xi):\,y_0\in B_{X^s}(0,r)\}
\end{equation}
are the fibers over $B_{X^c}(0,r)$,
or  equivalently over $\cE$ near $u_*$. We note that the fibers are $C^1$-manifolds, but they will depend only continuously on $\xi$, or equivalently on $\cE$.

The strategy of our approach can be summarized as follows:
given a base point $u_\ast + \xi+\phi(\xi)$ on the manifold $\cE$,
and an initial value $y_0\in X^s_\gamma$, we determine with the help of the implicit function
theorem an initial value $u_0$ and a solution $u(t)$ such that
$u(t)$ converges to the base point $u_\ast + \xi+\phi(\xi)$ exponentially fast.
Exponential convergence will be obtained by setting up the implicit function theorem
in a space of exponentially decaying functions.

After these heuristic considerations we can state our first main result, employing the notation introduced above.
\medskip
\begin{theorem}
\label{stable-foliation}
Consider \eqref{u-equation} under the assumption \eqref{AF}, and let $u_*\in \cE$ be a normally hyperbolic equilibrium. Suppose further that $A(u_\ast)$ has the property of $L_p$-maximal regularity.
Then there is a number $r>0$ and a continuous map
\begin{equation*}
\lambda^s:B_{X_\gamma^s}(0,r)\times B_{X^c}(0,r) \to X_\gamma\quad\text{with}\quad
\lambda^s(0,0)=u_*,
\end{equation*}
the \textbf{stable foliation}, such that
the solution $u(t)$ of \eqref{u-equation} with initial value $\lambda^s(y_0,\xi)$ exists on $\R_+$ and converges to $u_\infty:=u_*+\xi+\phi(\xi)$ in $X_\gamma$ exponentially fast as $t\to\infty$.
The image of $\lambda^s$ defines the stable manifold $\cM^s$ of \eqref{u-equation} near $u_*$.
\smallskip\\
Furthermore, for fixed $\xi\in B_{X^c}(0,r)$, the function
\begin{equation*}
\lambda^s_{\xi}:B_{X^s_\gamma}(0,r)\to X_\gamma, \quad\text{given by}\quad
\lambda^s_{\xi}(y_0)=\lambda^s(y_0,\xi),
\end{equation*}
defines the fibers $\cM^s_{\xi}:=\lambda^s_{\xi}(B_{X^s_\gamma}(0,r))$ of the foliation.
Moreover, we have
\begin{itemize}
\item[(i)]\,  an initial value $u_0\in X_\gamma$ near $u_*$ belongs to $\cM^s$ if and only if the solution $u(t)$ of \eqref{u-equation} exists globally on $\R_+$ and converges to some  $u_\infty\in\cE$
exponentially fast as $t\to\infty$;
\item[(ii)] \, $\lambda^s_{\xi}$ is of class $C^1$, and the derivative
$D_{y_0}\lambda^s$ is continuous, jointly in $(y_0,\xi)$;
\item[(iii)] \,the fibers are $C^1$-manifolds which are invariant under the semiflow generated
by \eqref{u-equation};
\item[(iv)] \, the tangent space of $\cM^s_{u_\infty}$ at $u_\infty\in\cE$ is precisely the projection of the stable part of the linearization of \eqref{u-equation} at $u_\infty$;
\item[(v)] \, if $u_*$ is normally stable then $\cM^s$ forms a neighborhood of $u_*$ in $X_\gamma$.
\end{itemize}
\end{theorem}
\begin{proof}
For fixed $\sigma\in(0,\omega]$ we define the function spaces
\begin{equation*}
\begin{aligned}
\FF_0^l(\sigma):&=\{f: e^{\sigma t}f\in L_p(\R_+;X_0^l)\},\\
\FF_1^l(\sigma):&=\{w: e^{\sigma t}w\in H^1_p(\R_+;X_0^l)\cap L_p(\R_+;X_1^l)\},
\quad l\in\{c,s,u\},
\end{aligned}
\end{equation*}
equipped with their natural norms.
Then we set
\begin{equation*}
\YY:=\YY(\sigma):= \FF_1^c(\sigma)\times\FF_1^s(\sigma)\times \FF_1^u(\sigma),
\quad Z:=X^s_\gamma\times X^c,
\end{equation*}
and we define
$
H_s:B_{\YY}(0,\rho)\times B_{Z}(0,\rho)\to \YY
$
 by \eqref{stableF}.
 Observe that $H_s$ is the composition of the substitution operator
\begin{equation*}
R(x,y,z,\xi)=[(R_c,R_s,R_u)(x+y+z+\xi+\phi(\xi)),\phi_s(\xi)]^{\sf T},
\end{equation*}
which maps $B_\YY(0,\rho)\times B_Z(0,\rho)$ into
$\FF_0^c(\sigma)\times\FF_0^s(\sigma)\times \FF_0^u(\sigma)\times X_\gamma^s$,
and the bounded linear operator
$$\LL(x,y,z,y_0,\xi,R_1,R_2,R_3,R_4) =
\left[\begin{array}{l} x(t)+\int_t^\infty R_1(\tau),d\tau\\
                       \vspace{-2mm}\\
                       y(t)-L_s(R_2,y_0-R_4)\\
                       \vspace{-2mm}\\
                       z(t)+\int_t^\infty e^{-A_u(t-\tau)}R_3(\tau)\,d\tau\\
                       \end{array}\right],$$
which maps
\begin{equation*}
\YY\times Z \times \FF_0^c(\sigma)\times\FF_0^s(\sigma)\times \FF_0^u(\sigma)\times X_\gamma^s
\quad \text{into}\quad \YY.
\end{equation*}
In order to see that the integral operator
$[R_1\mapsto \int_t^\infty R_1(\tau)\,d\tau]$ maps
$\FF_0^c(\sigma)$ into $\FF_1^c(\sigma)$
we set
$$ (KR_1)(t):=\int_t^\infty R_1(\tau)\,d\tau,\quad R_1\in \FF^c_0(\sigma).$$
Clearly,
$
e^{\delta t}(Kg)(t)=\int_t^\infty e^{\delta(t-\tau)}e^{\delta\tau}g(\tau)\,d\tau,
$
and
Young's inequality for convolution integrals readily yields
\begin{equation*}
K\in \cB\big(\FF^c_0(\sigma),\FF^c_1(\sigma)).
\end{equation*}
Similar arguments also apply for the term $\int_t^\infty e^{-A_u(t-\tau)}R_3(\tau)\,d\tau$.
\medskip\\
As $G$ is of class $C^1$, it is not difficult to see that
$R$ is $C^1$ with respect to the variables $(x,y,z,y_0)$,
but in general only continuous with respect to $\xi$.
This, in turn, implies that
$H_s$ is of class $C^1$ w.r.t.\ $(x,y,z,y_0)$,  and continuous in $\xi$.
In addition, as $G(0)=G^\prime(0)=0$, we obtain
\begin{equation*}
H_s(0,0)=0,\quad D_{(x,y,z)}H_s(0,0)=I_{\YY}.
\end{equation*}
Therefore, by the implicit function theorem, see \cite[Theorem 15.1]{D85},
there is a radius $r>0$ and a continuous map
$\Lambda^s:B_Z(0,r)\to \YY$ such that
$$H_s(\Lambda^s(y_0,\xi),(y_0,\xi))=0,\quad \mbox{ for all } \;  (y_0,\xi)\in B_Z(0,r),$$
and there is no other solution of
$H_s((x,y,z),(y_0,\xi))=0$ in the ball $B_\YY(0,r)\times B_Z(0,r)$.
Moreover, $\Lambda^s$ is also $C^1$ w.r.t.\ $y_0$, and one shows that
\begin{equation}
\label{Dy-Xi}
D_{y_0}\Lambda^s(0,0)w_0=[0,e^{-tA_s}w_0,0]^{\sf T}.
\end{equation}
 We may now continue as indicated in the heuristic considerations preceding
Theorem \ref{stable-foliation} to define the stable manifold $\cM^s$ and their fibers $\cM^s_{\xi}$ near $u_*$, see \eqref{xi-s}--\eqref{fiber-s}.
Clearly, the fibers are positively and negatively invariant under the semiflow, as \eqref{u-equation} is invariant concerning time-translations. As $\lambda_{\xi}^s$ is of class $C^1$, it is clear that the fibers are $C^1$-manifolds, parameterized over $X^s_\gamma$.
In fact, the fibers are diffeomorphic,  which can be seen by interchanging the roles
of $u_*$ and $u_\infty=u_*+\xi+\phi(\xi)$.
\smallskip\\
It follows from \eqref{Dy-Xi} that
$D_{y_0}\lambda^s(0,0)=I_{X^s_\gamma}$.
Interchanging the roles of $u_*$ and $u_\infty$, it becomes clear that the tangent space of the fiber $\cM_{\xi}^s$ at 0 is precisely the projection of the stable part of the linearization
$A_\infty=A(u_\infty) +[A^\prime(u_\infty)\cdot]u_\infty-F^\prime(u_\infty)$
of \eqref{u-equation} at $u_\infty\in\cE$,
yielding assertion (iv).
\smallskip\\
To obtain the characterization of $\cM^s$, observe that we proved in \cite{PSZ09} that there are balls $B_{X_\gamma}(u_*,r_0)$ and $B_{X_\gamma}(u_\ast,r_1)$ such that any solution of \eqref{u-equation} starting in $B_{X_\gamma}(u_*,r_0)$ and staying near $\cE$ stays in $B_{X_\gamma}(u_\ast,r_1)$ and converges to an equilibrium $u_\infty\in\cE$ exponentially fast. This implies that its initial value must belong to $\cM^s$ by uniqueness of $\Lambda^s$.

If $u_*\in\cE$ is normally stable then, as proved in \cite{PSZ09},
there are balls $B_{X_\gamma}(u_*,r_0)$ and $B_{X_\gamma}(u_\ast,r_1)$ such that any solution of \eqref{u-equation} starting in $B_{X_\gamma}(u_*,r_0)$
stays in $B_{X_\gamma}(u_\ast,r_1)$ and converges to an equilibrium $u_\infty\in\cE$ exponentially fast. This implies that $\cM^s$ forms a neighborhood of $u_*$ in $X_\gamma$,
thus establishing assertion (v) of the theorem.
\end{proof}
\section{The unstable foliation}
Our second main result concerning the unstable foliation of $\cE$ reads as follows.

\begin{theorem}
\label{unstable-foliation}
Consider \eqref{u-equation} under the assumption \eqref{AF}, and let $u_*\in \cE$ be a normally hyperbolic equilibrium. Suppose further that $A(u_\ast)$ has the property of $L_p$-maximal regularity.
Then there is a number $r>0$ and a continuous map
\begin{equation*}
\lambda^u:B_{X_\gamma^u}(0,r)\times B_{X^c}(0,r) \to X_\gamma\quad\text{with}\quad
\lambda^u(0,0)=u_*,
\end{equation*}
the \textbf{unstable foliation}, such that
the solution $u(t)$ of \eqref{u-equation} with initial value $\lambda^u(y_0,\xi)$ exists on $\R_-$ and converges to $u_\infty:=u_*+\xi+\phi(\xi)$ in $X_\gamma$ exponentially fast as $t\to -\infty$.
The image of $\lambda^u$ defines the unstable manifold $\cM^u$ of \eqref{u-equation} near $u_*$.
\\
Furthermore, for fixed $\xi\in B_{X^c}(0,r)$, the function
\begin{equation*}
\lambda^u_{\xi}:B_{X^u_\gamma}(0,r)\to X_\gamma, \quad\text{given by}\quad
\lambda^u_{\xi}(y_0)=\lambda^u(y_0,\xi),
\end{equation*}
defines the fibers $\cM^u_{\xi}:=\lambda^u_{\xi}(B_{X^u_\gamma}(0,r))$ of the foliation.
Moreover, we have
\begin{itemize}
\item[(i)]\,  an initial value $u_0\in X_\gamma$ near $u_*$ belongs to $\cM^u$ if and only if the solution $u(t)$ of \eqref{u-equation} exists globally on $\R_-$ and converges to some  $u_\infty\in\cE$
exponentially fast as $t\to -\infty$;
\item[(ii)] \, $\lambda^u_{\xi}$ is of class $C^1$, and the derivative
$D_{y_0}\lambda^u$ is continuous, jointly in $(y_0,\xi)$;
\item[(iii)] \,the fibers are $C^1$-manifolds which are invariant under the semiflow generated
by \eqref{u-equation};
\item[(iv)] \, the tangent space of $\cM^u_{\infty}$ at $u_\infty\in\cE$ is precisely the projection of the unstable part of the linearization of \eqref{u-equation} at $u_\infty$.
\end{itemize}
\end{theorem}
\begin{proof}
For fixed $\sigma\in(0,\omega]$ we define the function spaces
\begin{equation*}
\begin{aligned}
\FF_0^l(\sigma):&=\{f: e^{-\sigma t}f\in L_p(\R_-;X_0^l)\},\\
\FF_1^l(\sigma):&=\{w: e^{-\sigma t}w\in H^1_p(\R_-;X_0^l)\cap L_p(\R_-;X_1^l)\},
\quad l\in\{c,s,u\},
\end{aligned}
\end{equation*}
equipped with their natural norms.
We set
\begin{equation*}
\YY:=\YY(\sigma):= \FF_1^c(\sigma)\times\FF_1^s(\sigma)\times \FF_1^u(\sigma),
\quad Z:=X^u_\gamma\times X^c,
\end{equation*}
and we define
$
H_u:B_{\YY}(0,\rho)\times B_{Z}(0,\rho)\to \YY
$
by
\begin{align*}
\label{unstableF}
H_u((x,y,z),(z_0,\xi))(t)=
\left[\begin{array}{l}
x(t)-\int^t_{-\infty} R_c(x(\tau),y(\tau),z(\tau),\xi)\,d\tau\\
\vspace{-2mm}\\
y(t)-\int_{-\infty}^t e^{-A_s(t-\tau)}R_s(x(\tau),y(\tau),z(\tau),\xi)\,d\tau\\
\vspace{-2mm}\\
z(t)- L_u(R_u(x,y,z,\xi), z_0-\phi_u(\xi))\\
\end{array}\right],
\end{align*}
where $t<0$ and $w=L_u(f,z_0)$ denotes the unique solution of the backward problem
$$ \dot{w}(t)+A_u w(t)=f(t),\; t\leq 0,\quad w(0)=z_0.$$
Again, we have $H_u(0,0)=0$, $H_u$ is of class $C^1$ w.r.t.\ $(x,y,z,z_0)$,
 but only continuous in $\xi$, and $D_{(x,y,z)}H_u(0,0)=I$.
As in the previous proof, the implicit function theorem yields a map
$\Lambda^u:(z_0,\xi)\mapsto (x,y,z)$ such that
$H_u(\Lambda^u(z_0,\xi),(z_0,\xi))=0.$
Then
$$\lambda^u:(z_0,\xi)\mapsto u(0)=u_*+v(0)$$
yields the foliation of the unstable manifold $\cM^u$ near $u_*$,
and
$$\cM_{\xi}^u:=\{\lambda^u(z_0,\xi):\,z_0\in B_{X^u}(0,r)\}$$
are the fibers over $B_{X^c}(0,r)$, or equivalently over $\cE$ near $u_*$.
Note that the fibers $\cM_{\xi}^u$ are $C^1$-manifolds,
but they depend only continuously on $\xi$ or equivalently on $\cE$.

The reminder of the proof then follows along similar lines as that of
Theorem \ref{stable-foliation} and is therefore left to the interested reader.
\end{proof}

Note that the fibers can be extended to global fibers following the backward or forward flow as long as it exists. Laterally, i.e.\ along the direction of $\cE$, the foliation can be extended up to equilibria $u_\#\in\cE$ which are no longer normally hyperbolic.
\smallskip\\
Concerning regularity of the foliation we note that the implicit function theorem yields the following result.
\medskip
\begin{corollary}
Under the assumptions of Theorem \ref{stable-foliation} and Theorem \ref{unstable-foliation}
the following regularity result is valid: if
$$(A,F)\in C^k(V,\cB(X_1,X_0)\times X_0)$$
then the foliations satisfy $\Lambda^l\in C^{k-1}$ and the fibers $\lambda^l_{\xi}$ are
of class $C^k$, for all
$k\in\N\cup\{\infty,\omega\}$, where $l\in\{s,u\}$ and $C^\omega$ means real analytic.
\end{corollary}
\begin{proof}
It follows from the regularity assumptions  that
$G\in C^k(V,X_0)$. We can then conclude from the second line in \eqref{equilibria-phi}
that
$\phi\in C^k(B_{X^c}(0,\rho_0),X^s_1\oplus X^u_1).$
Indeed, for this it suffices to observe that $\phi$ is implicitly defined
by the second line in \eqref{equilibria-phi}, as $A_s$ and $A_u$ are invertible.
The assertions follow now from the proof of Theorems \ref{stable-foliation} and \ref{unstable-foliation}.
\end{proof}
\bigskip
\noindent
{\bf Acknowledgment:}
J.P. and M.W. express their thanks to the Department of Mathematics, Vanderbilt University, for financial support and kind hospitality.


\end{document}